\documentclass[12pt,reqno]{amsart}
\usepackage{amsmath, amsthm, amssymb, stmaryrd}
\usepackage{hyperref}
\usepackage{enumerate}
\usepackage{url}
\usepackage{color}
\usepackage{comment}

\let\OLDthebibliography\thebibliography
\renewcommand\thebibliography[1]{
  \OLDthebibliography{#1}
  \setlength{\parskip}{0pt}
  \setlength{\itemsep}{0pt plus 0.3ex}
}

\usepackage{mathtools}

\usepackage{multirow, array}
\usepackage{placeins}

\usepackage{caption} 

\advance \topmargin by -\headheight
\advance \topmargin by -\headsep
     
\evensidemargin \oddsidemargin
\newtheorem{theorem}{\bf Theorem}[section]

\newtheorem{lemma}[theorem]{\bf Lemma}

\newtheorem{conjecture}[theorem]{\bf Conjecture}
\newtheorem{question}[theorem]{\bf Question}

\numberwithin{equation}{section}

\newcommand*\wrapletters[1]{\wr@pletters#1\@nil}
\def\wr@pletters#1#2\@nil{#1\allowbreak\if&#2&\else\wr@pletters#2\@nil\fi}

\begin{document}
\title[Additive and multiplicative Sidon sets]{Additive and multiplicative Sidon sets}

\author{Oliver Roche-Newton}
\address{Johann Radon Institute for Computational and Applied Mathematics (RICAM), Linz, Austria}
\email{o.rochenewton@gmail.com}

\author{Audie Warren}
\address{Johann Radon Institute for Computational and Applied Mathematics (RICAM), Linz, Austria}
\email{audie.warren@oeaw.ac.at}

\thanks{}
\date{}
\begin{abstract} 
We give a construction of a set $A  \subset \mathbb N$ such that any subset $A' \subset A$ with $|A'| \gg |A|^{2/3}$ is neither an additive nor multiplicative Sidon set. In doing so, we refute a conjecture of Klurman and Pohoata.
\end{abstract}
\maketitle

\section{Introduction}
\label{intro}

A set $A \subset \mathbb R$ is a \textit{Sidon set} if the only solutions to the energy equation
\[
a +b=c+d \text{ such that } a,b,c,d \in A
\]
are the trivial solutions whereby $\{a,b\}=\{c,d\}$.  Sidon sets are important and widely studied objects in Additive Combinatorics; see \cite{tv} for background on this topic and \cite{OB} for a thorough survey.

Given an arbitrary finite set $A \subset \mathbb R$, we define $s_+(A)$ to be the size of the largest Sidon set $A'$ such that $A' \subset A$. The problem of understanding the behaviour of $s_+([N])$ has attracted particular interest, and it is known that $s_+([N])= \Theta(N^{1/2})$. See \cite{Gow} for a nice introduction to this question.

Komlos, Sulyok and Szemer\'{e}di \cite{KSS} proved (as part of a much more general result about sets avoiding certain linear configurations) that, up to constant factors, $s_+(A)$ is minimised among sets of size $N$ by the case when $A=\{1,2,\dots,N\}$. That is, they proved the bound
\begin{equation} \label{KSS}
s_+(A) \gg |A|^{1/2}
\end{equation}
holds for any\footnote{In fact, this was proved in \cite{KSS} only for sets of integers, but it has later been established that the same result holds over the reals; see for instance \cite[Lemma 2.2]{Raz}.} finite set $A \subset \mathbb R$.

Similarly, a \textit{multiplicative Sidon set} is a set $A \subset \mathbb R$ such that the only solutions to 
\[
ab=cd \text{ such that } a,b,c,d \in A
\]
are the trivial solutions with $\{a,b\}=\{c,d\}$. Given an arbitrary finite set $A \subset \mathbb R$, we define $s_*(A)$ to be the size of the largest multiplicative Sidon set $A'$ such that $A' \subset A$.  Applying the estimate \eqref{KSS} to the set $\log A=\{\log a : a \in A\}$, it follows that
\begin{equation} \label{KSS*}
s_*(A) \gg |A|^{1/2}
\end{equation}
for any finite $A \subset \mathbb R$.

In the spirit of the sum-product problem, one may ask the question of whether it is guaranteed that, for a fixed set $A \subset \mathbb R$, at least one of the inequalities \eqref{KSS} and \eqref{KSS*} is improvable. This question was indeed asked by Klurman and Pohoata \cite{KP}, who made the following conjecture.
\begin{conjecture} \label{conj:main}There exists a constant $c>0$ such that for all finite $A \subset \mathbb R$,
\begin{equation} \label{conj}
\max \{s_+(A), s_*(A)\} \gg  |A|^{1/2 + c}.
\end{equation}
\end{conjecture}
Klurman and Pohoata in fact also made the stronger conjecture that, for all $\epsilon>0$,
\begin{equation} \label{conj2}
\max \{s_+(A), s_*(A)\} \gg  |A|^{1- \epsilon}.
\end{equation}
The main result of this note is a construction disproving \eqref{conj2}.

\begin{theorem} \label{thm:main}
For all $N \in \mathbb N$, there exists a set $A \subset \mathbb N$ with $|A| \geq N$ and such that $s_+(A) \ll |A|^{1/2}\log |A|$ and $s_*(A) \ll |A|^{2/3}$. In particular,
\[
\max \{s_+(A), s_*(A)\} \ll |A|^{2/3}.
\]
\end{theorem}

\subsection{Notation} Throughout the paper, the standard notation
$\ll,\gg$ and respectively $O$ and $\Omega$ is applied to positive quantities in the usual way. That is, $X\gg Y$, $Y \ll X,$ $X=\Omega(Y)$ and $Y=O(X)$ all mean that $X\geq cY$, for some absolute constant $c>0$. If both $X \ll Y$ and $Y \ll X$ hold we write $X \approx Y$, or equivalently $X= \Theta(Y)$. All logarithms are in base $e$. For a set $A$, the notation $A(\cdot)$ is used for its characteristic function.
\section{Construction}

\begin{proof}[Proof of Theorem \ref{thm:main}]

For a sufficiently large integer $n$, define $P$ to be the set of all primes less than or equal to $n$, and define 
\[
Q=\left \{n<q \leq \frac{n^2}{\log n} : q \text{ prime} \right \} .
\] 
The construction is simply the product set
\[
A:=P \cdot Q=\left \{p\cdot q : p,q \text{ prime, } p\leq n < q \leq \frac{n^2}{\log n} \right \}.
\]
$A$ is a fairly dense subset of the set $\{1,2,\dots,n^3/\log n\}$, which implies that $A+A$ is small and thus $s_+(A)$ is also small. Indeed, by the Fundamental Theorem of Arithmetic and the Prime Number Theorem, 
\[
|A| = |P||Q| \approx \frac{n^3}{\log^3n}. 
\]
Since $A+A \subset \left \{1,\dots,\frac{2n^3}{\log n} \right \}$, it follows that 
\begin{equation} \label{sumset}
|A+A| \ll |A|\log^2|A|.
\end{equation} 
Furthermore, since a Sidon set of size $m$ contained in $A$ gives rise to $\binom{m}{2}+m$ distinct sums in $A+A$, it follows that
\[
s_+(A) \ll |A+A|^{1/2} \ll |A|^{1/2} \log |A|.
\]

It remains to prove the upper bound
\begin{equation} \label{aim}
s_*(A) \ll \frac{n^2}{\log^2 n} \approx |A|^{2/3}.
\end{equation}
 To do this, we appeal to an argument Erd\H{o}s \cite{Erdos} used in order to establish the existence of a set with very small additive energy but still having fairly small $s_+(A)$. What follows can be viewed as an adaptation of this argument in the multiplicative setting.

We view the elements of $A$ as edges of a complete bipartite graph $G=(V,A)$ with vertex set $V =P \sqcup Q$, and with the edge between $p \in P$ and $q \in Q$ labelled by the product $pq$. 

Suppose that $E \subset A$ is a multiplicative Sidon set. Then $E$ does not give rise to any copies of the cycle $C_4$ in $G$. Indeed, suppose that we have four edges in $E$ which form a cycle. Then the cycle must be of the form $pq, qp', p'q',q'p$. But if all of these four edges are in $E$ then we have 
\[
(pq)(p'q')=(p'q)(pq'),
\]
which contradicts the fact that $E$ is a multiplicative Sidon set.

On the other hand, any $E \subset A$ which does not contain a copy of $C_4$ has size $|E| \ll \frac{n^2}{\log^2 n}$ by the K\H{o}v\'{a}ri-S\'{o}s-Tur\'{a}n Theorem. For completeness, we give a quick proof below using an application of Cauchy-Schwarz:

\begin{align*}
|E|^2&= \left( \sum_{q \in Q} \sum_{p \in P} E(pq) \right )^2
\\& \leq |Q| \sum_{q \in Q} \left (\sum_{p \in P} E(pq) \right )^2
\\ & \ll \frac{n^2}{\log^2 n} \sum_{p,p' \in P} \sum_{q \in Q} E(pq)E(p'q)
\\&=\frac{n^2}{\log^2 n} \left ( |E|+\sum_{p,p' \in P : p \neq p'} \sum_{q \in Q} E(pq)E(p'q) \right )
\\& \leq \frac{n^2}{\log^2 n} \left ( |E|+|P|^2 \right ) \ll  \frac{n^2}{\log^2 n} \left ( |E|+\frac{n^2}{\log^2 n} \right ).
\end{align*}
If the sum $|E|+\frac{n^2}{\log^2 n} $ is dominated by the second term then we are done. Therefore we may assume otherwise, and the last inequality gives
\[
|E|^2 \ll \frac{n^2}{\log^2 n} \left ( |E|+\frac{n^2}{\log^2 n} \right ) \ll \frac{n^2}{\log^2 n} |E|.
\]
We have proved that any subset of size greater than $Cn^2/\log^2 n$, for some constant $C$, contains a $C_4$, and thus is not a multiplicative Sidon set. This proves \eqref{aim} and thus completes the proof of the theorem.
\end{proof}

\section{Connections to additive and multiplicative energy}
\subsection{The relationship between $s_+(A)$ and additive energy}

The additive energy of $A$, denoted $E_+(A)$, is the number of ordered quadruplues  \linebreak $(a,b,c,d) \in A \times A \times A \times A$ such that
\[
a+b=c+d.
\]
Note that $E_+(A)$ also counts the trivial solutions, and so we have $E_+(A) \geq |A|^2$.
One may suspect that if $A$ has small additive energy then $s_+(A)$ must be large. This is true to some extent, as  the following lemma shows.

\begin{lemma} \label{lem:random} For any finite $A \subset \mathbb R$
\[
s_+(A) \gg \frac{|A|^{4/3}}{(E_+(A))^{1/3}}.
\]
\end{lemma}

The proof of Lemma \ref{lem:random} uses a simple probabilistic argument, and is implicit in the work of Alon and Erd\H{o}s \cite{AE}. The multiplicative analogue of Lemma \ref{lem:random} (with $s_*(A)$ and $E_*(A)$ in place of $s_+(A)$ and $E_+(A)$ respectively) holds via the same reasoning. This lemma overtakes the bound \eqref{KSS} when $E_+(A) \ll |A|^{5/2}$. Combining these two results together, we have
\begin{equation}\label{combo}
s_+(A) \gg  \max \left \{ \frac{|A|^{4/3}}{(E_+(A))^{1/3}}, |A|^{1/2} \right \}.
\end{equation}

The purpose of this section is to make the observation that the combined bound \eqref{combo} is in fact optimal. This is a consequence of the following result of Kohayakawa, Lee and R\"{o}dl \cite{KLR} (we only state the parts of the result which are relevant to our analysis).

\begin{theorem}[\cite{KLR}, Theorem 1.1] \label{thm:KLR}
Let $1/3 \leq a \leq 1$ be a fixed constant and $m=n^a(1+o(1))$. Let $B \subseteq [n]$ be a random set of size $m$ (i.e. $B$ is chosen randomly from all sets of size $m$). Then, almost surely,
\begin{enumerate}
    \item if $1/3 \leq a \leq 2/3$ then $s_+(B) =n^{1/3+o(1)}$,
    \item if $2/3 \leq a \leq 1$ then $s_+(B)=|B|^{1/2+o(1)}$.
\end{enumerate}
\end{theorem}

The $o(1)$ terms in the exponents suppress logarithmic factors. A more precise version of this statement is given in \cite[Theorems 2.2-2.4]{KLR}.

For a random set $B \subset [n]$ of size $m$, the expected size of $E_+(B)$ is 
\begin{equation}\label{energy}
cn^3 \frac{\binom{n-4}{m-4}}{\binom{n}{m}}=c'\frac{m^4}{n}.
\end{equation}

Part (2) of Theorem \ref{thm:KLR} implies that, for a random set $B \subset [n]$ of size \linebreak $n^{2/3+o(1)}\leq|B| \leq n$, the estimate $s_+(B) \gg |B|^{1/2}$ contained in \eqref{combo} is tight, up to the $o(1)$ factor. By \eqref{energy} these sets will typically have far from maximal additive energy, and taking $m$ close to $n^{2/3}$ yields a set $B$ with $E_+(B) \approx |B|^{5/2}$ and $s_+(B)$ very small.

Furthermore we have, with high probability,
\[
\frac{ |B|^{4/3}}{E_+(B)^{1/3}} \approx n^{1/3},
\]
and thus taking $B$ to be a random subset of $[n]$ with size $n^a$, $1/3 \leq a \leq 2/3$, the estimate from Lemma \ref{lem:random} matches the information $s_+(B)=n^{1/3 +o(1)}$ given by part (1) of Theorem \ref{thm:KLR}, up to the $o(1)$ factor.


\subsection{A variant of the Klurman-Pohoata Conjecture}

Conjecture \ref{conj:main} could be viewed as a possible line of attack for the sum-product problem, since the inequalities
\[
|A+A| \gg s_+(A)^2 ,\,\,\,\,\,\, |AA| \gg s_*(A)^2 
\]
imply that a positive answer to Conjecture \ref{conj} would give a non-trivial sum-product bound. Unfortunately, the construction in Theorem \ref{thm:main} shows that even the best possible result in this direction would not yield quantitative improvements to known sum-product inequalities.

However, it is not necessary that a set be additively/multiplicatively Sidon in order for it to determine many sums/products. A weaker assumption that the additive or multiplicative energy is small would suffice, in light of the usual Cauchy-Schwarz energy bounds
\[
E_+(A) \geq \frac{|A|^4}{|A+A|},\,\,\,\,\,\,\,E_*(A) \geq \frac{|A|^4}{|AA|}.
\]
With this in mind, we propose a variant of Conjecture \ref{conj:main}. For a set $A \subset \mathbb R$, define\footnote{The multiplicative constant $2$ in the definition of $t_+(A)$ is not particularly important, but is chosen so that at least half of the contributions to $E_+(A')$ come from the $|A'|^2$ trivial solutions.}  $t_+(A)$ to be the size of the largest subset $A' \subset A$ such that 
\[
E_+(A') <2|A'|^2.
\]
Similarly, $t_*(A)$ denotes the size of the largest $A' \subset A$ such that $E_*(A') <2|A'|^2$.

\begin{question} \label{question}
For what value of $\kappa>0$ is it true that any finite set $A \subseteq \mathbb R$ satisfies
\[
\max \{ t_+(A), t_*(A)\} \geq |A|^{\kappa}.
\]
\end{question}
The bound
\begin{equation} \label{basic}
\max \{ t_+(A), t_*(A)\} \gg |A|^{1/2}
\end{equation}
follows immediately from the result of Komlos, Sulyok and Szemer\'{e}di \cite{KSS}, and can also be proved by much simpler means. For this question, an improvement of \eqref{basic} follows from the Balog-Wooley Theorem. We will use the following result of Rudnev, Shkredov and Stevens \cite{EnergyVariant}.

\begin{theorem} \label{thm:BW}
Let $A \subset \mathbb R$. Then there exists $A' \subset A$ with $|A'| \geq  |A|/2$ and
\[
\min \{E_+(A'), E^*(A')\} \ll |A|^{11/4}. 
\]
\end{theorem}

We prove the following.
\begin{theorem} 
For any $A \subset \mathbb R$,
\[
\max \{ t_+(A), t_*(A)\}  \gg |A|^{5/8}.
\]
\end{theorem}
\begin{proof}
Apply Theorem \ref{thm:BW} to obtain a subset $A'$ and suppose that $E_+(A') \leq C|A|^{11/4}$. Let $A''$ be a $p$-random subset of $A'$, with $p=\frac{1}{100C^{1/2}|A|^{3/8}}$. The expected size of $A''$ is at least $p|A|/2$. Let $E_+^0(A'')$ denote the number of non-trivial solutions to 
\[
a_1+a_2=a_3+a_4 \text{ such that } a_i \in A''.
\]
The expected value of $E_+^0(A'')$ is at most
\[
Cp^4|A|^{11/4}+p^3|A|^2.
\]
By Markov's inequality, the probability that 
\begin{equation} \label{event1}
E_+^0(A'') \geq \frac{1}{100}p^2|A|^2
\end{equation}
is at most $1/10$. By Chebyshev, the probability that
\begin{equation} \label{event2}
|A''| <\frac{1}{10}p|A|
\end{equation}
is at most $1/10$. Therefore, with positive probability both events \eqref{event1} and \eqref{event2} do not occur. It therefore follows that there exists a set $A'' \subset A$ such that
\[
|A''| \gg |A|^{5/8}
\]
and
\[
E_+^0(A'') < \frac{1}{100}p^2|A|^2  \leq |A''|^2.
\]
In particular, $E_+(A'') < 2|A''|^2$, and it follows that
\[
t_+(A) \gg |A|^{5/8}. 
\]

If instead we have $E^*(A') \leq C|A|^{11/4}$ at the outset, we can run the same argument in the multiplicative setting and conclude that
\[
t_*(A) \gg |A|^{5/8}.
\]

\end{proof}

Note that a small improvement to this result can be obtained by instead applying a quantitative improvement to Theorem \ref{thm:BW}, due to Shakan \cite{Shakan}, but we have avoided this above in order to simplify the presentation. It can be calculated that this change to the proof results in the improved bound
\[
\max \{ t_+(A), t_*(A)\}  \gg |A|^{33/52-o(1)}.
\]
Note that $33/52=5/8+1/104$.

The construction in the proof of Theorem \ref{thm:main} does not give a non-trivial construction for Question \ref{question}. However, the original construction of Balog and Wooley \ref{BW} does yield the following.

\begin{theorem} \label{thm:tupper}
There exists a finite set $A \subset \mathbb R$ such that
\[
\max \{ t_+(A), t_*(A)\}  \ll |A|^{5/6}
\]
\end{theorem}

\begin{proof}
Define
\[
A:= \{ (2i-1)2^j : i \in [N^2], j \in [N]\}.
\]
This set is a union of $N$ arithmetic progressions of length $N^2$, and since any solution to $(2i-1)2^j = (2i'-1)2^{j'}$ implies that $i=i'$ and $j=j'$, we have that $|A| = N^3$. Suppose that $|A'| \geq C|A|^{5/6} = CN^{5/2}$ for a sufficiently large constant $C$. Note that since
$$AA \subseteq \{ (2i-1)2^j : i \in [2N^4], j \in [2N]\} $$
we have $|A'A'| \leq |AA| \leq 4N^5$. Applying the Cauchy Schwarz energy bound then yields
$$E_*(A') \geq \frac{|A'|^4}{|A'A'|} \geq \frac{C^2}{4}|A'|^2.$$
Therefore, as long as $C$ is sufficiently large we must have $t_*(A) \leq C|A|^{5/6}$.

To show that $t_+(A) \ll |A|^{5/6}$, we define
$$A_j = \{ (2i-1)2^j : i \in [N^2]\}$$
noting that we have $A = \bigcup_{j\in [N]}A_j$, and that each $A_j$ is an arithmetic progression with $|A_j| = N^2$. Let us define
$$I := \left\{ j \in [N] : |A' \cap A_j| \geq 2|A'|^{1/3}N^{2/3}\right\}.$$ 
We then have
\begin{align*}
|A'|        & = \sum_{j \in [N]} |A' \cap A_j| \\
            & = \sum_{j \in I} |A' \cap A_j| + \sum_{j \notin I} |A' \cap A_j|  \\
            & \leq  \sum_{j \in I} |A' \cap A_j| +  2|A'|^{1/3}N^{5/3}.
\end{align*}
We now bound the energy of $A'$.
\begin{align*}
    E(A') & \geq \sum_{j \in I}E(A' \cap A_j) \\
    & \geq \sum_{j \in I} \frac{|A' \cap A_j|^4}{|A_j + A_j|} \\ & \gg \frac{1}{N^2} \sum_{j \in I}|A' \cap A_j|^4 \\
    & \geq 8|A'| \sum_{j \in I}|A' \cap A_j| \\
    & \geq 8|A'|^2 - 16|A'|^{4/3}N^{5/3}.
\end{align*}
From the final inequality it follows that $E(A') \geq 2|A'|^2$ as long as the constant $C$ is sufficiently large. Therefore $t_+(A) \ll |A|^{5/6}$, proving the result.
\end{proof}
Note that the bound $|AA| \leq N^5$ above is slightly wasteful and can be improved by a logarithmic factor. Using this fact and slightly rebalancing the construction of $A$ above gives the small improvement
\[
\max \{ t_+(A), t_*(A)\}  \ll \frac{|A|^{5/6}}{\log^{c}|A|}
\]
for some constant $c>0$.

\subsection{Connection to Balog-Wooley decomposition}

Conjecture \eqref{conj} is also connected to a question of Balog and Wooley \cite{BW}: For which $\kappa>0$ is it true that any finite set $A \subseteq \mathbb R$ can always be partitioned into two subsets $B$ and $C$, such that
\begin{equation}\label{BW} \max \{ E_+(B), E_*(C)\} \ll |A|^{3 - \kappa}.\end{equation}
A construction is given in \cite{BW} (and which was repeated here in the proof of Theorem \ref{thm:tupper}) proving that any such result must have $\kappa < 2/3$. Although the authors do not go as far as to conjecture that an exponent of $7/3$ is attainable in \eqref{BW}, a similar conjecture is stated in \cite{EnergyVariant}, where the authors conjecture that for all finite $A \subseteq \mathbb R$, there exists a subset $A' \subseteq A$ with $|A'| \geq \frac{|A|}{2}$ such that 
\begin{equation}\label{BWvariant} \min \{ E_+(A'), E_*(A') \} \ll |A|^{7/3+o(1)}.\end{equation}
A positive answer to this conjecture would imply a positive answer to Conjecture \ref{conj:main}. Indeed, utilising Lemma \ref{lem:random} (or possibly its multiplicative analogue), the subset $A'$ satisfies
$$|A|^{5/9-o(1)}\ll |A'|^{5/9-o(1)} \ll \max\{s_+(A'), s_*(A') \} \leq \max\{s_+(A), s_*(A) \},$$
thus giving Conjecture \ref{conj:main} with $c = 1/18-o(1)$. Weaker variants of conjecture \eqref{BWvariant} also imply Conjecture \ref{conj:main}, as long as the exponent is at most $5/2 - \epsilon$. Stating this in the contrapositive, any construction disproving Conjecture \ref{conj:main} also disproves strong Balog Wooley type conjectures.

\section{Further remarks}
\subsection{Small sum set implies large $s_*(A)$?}

Another conjecture of Klurman and Pohoata was the following: for all $\epsilon>0$, and any set $A \subset \mathbb R$
\begin{equation} \label{conj3}
|A+A| \leq K|A| \Rightarrow  s_*(A) \gg_{K,\epsilon} |A|^{1-\epsilon}.
\end{equation} 
By combining the multiplicative analogue of Lemma \ref{lem:random} with Solymosi's \cite{S} bound on the multiplicative energy via sumsets, they proved that
\begin{equation} \label{optimal}
|A+A| \leq K|A| \Rightarrow s_*(A) \gg \frac{|A|^{2/3}}{K^{2/3} (\log |A|)^{1/3}}.
\end{equation}
The construction given in the proof of Theorem \ref{thm:main} disproves the conjecture \eqref{conj3}, at least in the range when $K= \log^c |A|$. Indeed, we recorded in $\eqref{sumset}$ that the construction satisfies $|A+A| \ll |A|\log^2|A|$. Moreover, since we have $s_*(A) \ll  |A|^{2/3}$, this construction also shows that the bound \eqref{optimal} is in fact optimal in this range, up to logarithmic factors.

\subsection{Other constructions}

Theorem \ref{thm:main} was obtained independently by Green and Peluse (private communication). We thank them for sharing their construction with us. The construction is similar to the set $A$ defined in the proof of Theorem \ref{thm:main}. Paraphrasing slightly, they define a set 
\[
B=\{p_1p_2p_3 : 1 \leq  p_i \leq N, \text{ $p_i$ prime} \}.
\]
The set $B$ is a dense enough in $[N^3]$ to ensure that its sum set is small and thus $s_+(B)$ is small. On the other hand, and similarly to our proof of Theorem \ref{thm:main}, an application of the Cauchy-Schwarz inequality can be used to show that any subset of $B$ larger than $C|B|^{2/3}$ must contain four elements of the form 
\[pqr, p'qr, pq'r', p'q'r',\]
which gives rise to a non-trivial multiplicative energy solution $(pqr)(p'q'r')=(p'qr)(pq'r')$.

In a forthcoming paper, Shkredov \cite{Shk} gives another construction of a set with $\max \{s_+(A),s_*(A) \} \ll |A|^{2/3}$. His construction is somewhat different in nature, and comes from taking $A$ to be a sum set of carefully chosen geometric progressions.

\section*{Acknowledgements}
The authors were supported by the Austrian Science Fund FWF Projects P 30405-N32 and P 34180. We are very grateful to Cosmin Pohoata for bringing this problem to our attention, and for several helpful discussions. Thanks also to Ben Green, Oleksiy Klurman, Sarah Peluse, Ilya Shkredov and Sophie Stevens for helpful discussions.

\providecommand{\bysame}{\leavevmode\hbox to3em{\hrulefill}\thinspace}

\end{document}